\documentclass[a4paper,10pt,notitlepage,reqno]{article}
\usepackage[english]{babel}
\usepackage{amssymb,amsthm,amsmath} 
\usepackage{mathtools}
\usepackage{mathrsfs}
\usepackage{enumerate}
\usepackage{authblk}
\usepackage{cite}
\usepackage[symbol]{footmisc}
\usepackage{color}
\usepackage{geometry}
\usepackage[colorlinks=true, linkcolor=red, citecolor=blue, urlcolor=blue,hyperfootnotes=false]{hyperref}
\theoremstyle{plain} 
\newtheorem{theorem}{Theorem}[section]
\newtheorem{lemma}{Lemma}[section]

\newtheorem{corollary}{Corollary}[section]

\theoremstyle{definition}

\theoremstyle{remark}
\newtheorem{remark}{Remark}[section]

\DeclareMathOperator{\divergenza}{div}
\renewcommand{\div}{\divergenza}

\newcommand{\R}{\mathbb{R}}

\newcommand{\C}{\mathbb{C}}

\newcommand{\N}{\mathbb{N}}
\newcommand{\normeq}[1]{{\left\vert\kern-0.25ex\left\vert\kern-0.25ex\left\vert #1 
    \right\vert\kern-0.25ex\right\vert\kern-0.25ex\right\vert}}

{\left\lbrace\begin{array}{r@{\hspace{1mm}}ll}}%
{\end{array}\right.}

\title{\textbf{Spectral enclosures for the damped elastic wave equation}}
\author[1]{Biagio Cassano}
\author[2]{Lucrezia Cossetti} 
\author[3]{Luca Fanelli}		
\affil[1]{Dipartimento di Matematica e Fisica, Università degli Studi
  della Campania  ``Luigi Vanvitelli", Viale Lincoln 5, 81100 Caserta, Italy; biagio.cassano@unicampania.it}
\affil[2]{Fakult\"{a}t f\"{u}r Mathematik, Institut f\"{u}r Analysis, Karlsruher Institut f\"{u}r Technologie (KIT), Englerstra{\ss}e 2, 76131 Karlsruhe, Germany; lucrezia.cossetti@kit.edu}
\affil[3]{Ikerbasque $\&$ Departamento de Matematicas, Universidad del Pa\'is Vasco/Euskal Herriko Unibertsitatea (UPV/EHU), Aptdo. 644, 48080, Bilbao, Spain; luca.fanelli@ehu.es}
\begin{document}

\date{\small 17 August 2021}

\maketitle



\begin{abstract}
	\noindent
	In this paper we investigate spectral properties of the damped elastic wave equation. Deducing a correspondence between the eigenvalue problem of this model and the one of Lamé operators with non self-adjoint perturbations, we provide quantitative bounds on the location of the point spectrum in terms of suitable norms of the damping coefficient.
\end{abstract}

\section{Introduction}
This paper is concerned with the \emph{damped elastic wave equation}
\begin{equation}\label{eq:damped}
	u_{tt} + a(x)u_t-\Delta^\ast u=0, 
	\qquad (x,t)\in \R^d\times (0,\infty), 
\end{equation}
Here $a\colon \R^d \to \C^{d\times d}$ denotes the damping coefficient assumed to be a (possibly) non hermitian matrix. We shall make the standard assumption of a bounded damping, \emph{i.e.} $a\in L^\infty(\R^d)^d.$ The symbol $-\Delta^\ast$ is used to denote the Lamé operator of elasticity which is a matrix-valued differential operator acting, w.r.t. the spacial variable $x\in \R^d$ on smooth vector fields as
\begin{equation}\label{eq:Lame}
	-\Delta^\ast u=-\mu\Delta u -(\lambda+\mu)\nabla \div u,
	\qquad 
	u\in C^\infty_0(\R^d)^d:=C^\infty_0(\R^d;\C^d).
\end{equation}   
The material-dependent Lamé parameters $\lambda, \mu\in \R$ are assumed to satisfy the ellipticity condition
\begin{equation}\label{eq:ellipticity}
	\mu>0, 
	\quad
	\lambda+ \mu\geq 0.
\end{equation}

\noindent
It is customarily to write the second-order evolution system~\eqref{eq:damped} as a doubled first-order system introducing the vector field $U=(u, u_t)^{T}.$ Then~\eqref{eq:damped} can be rewritten as $U_t=A^\ast_a U,$ where $A_a^\ast$ is the $2d\times 2d$ matrix-valued \emph{damped elastic wave operator} defined as
\begin{equation}\label{eq:damped-operator}
	A_a^\ast:=
	\begin{pmatrix}
		0 &1\\
		\Delta^\ast & -a
	\end{pmatrix},
	\qquad \mathcal{D}(A_a^\ast):=H^2(\R^d)^d\times \dot{H}^1(\R^d)^d. 
\end{equation}   
The damped elastic wave equation~\eqref{eq:damped} and the corresponding damped operator~\eqref{eq:damped-operator} have attracted considerable attention in the last decades. In the constant coefficient case, namely $a(x)=\alpha,$ $\alpha>0,$ Bocanegra-Rodr\'iguez et al.~\cite{B-R2020} considered the longtime dynamics of this semilinear model in the presence of nonlinear structural forcing terms and external forces: they proved well-posedness \emph{\`a la} Hadamard and established the existence of finite dimensional global attractors together with the upper semicontinuity thereof. Energy decay results in relation with stability properties of solutions to this elastic model have been also deeply investigated. In~\cite{BD2013} Bchatnia and Daoulatli obtained a general energy decay estimate in a three dimensional bounded domain in the presence of localized nonlinear damping and an external force. By adding viscoelastic dissipation of memory type Bchatnia and Guesmia~\cite{BG2014} established a more general energy decay. Different viscoelastic dissipations have been considered in~\cite{Mustafa1,Mustafa2}.
Strong stability of Lamé systems with fractional order boundary damping were studied by Benaissa and Gaouar in~\cite{BG2019}.

For the undamped elastic wave equation, more commonly known as \emph{Navier equation}, a more varied bibliography is available. In~\cite{B_F-G_P-E_R_V} Barcel\'o et al.~proved uniform resolvent estimates (Limiting Absorption Principle) for this model. With this stationary tool at hands they also proved a priori averaged estimates for the corresponding Cauchy problem. The resolvent estimates in~\cite{B_F-G_P-E_R_V} were generalized in~\cite{Cossetti2019} and then improved in~\cite{KLS2021}, where a sharp result (analogous to the one available for the Laplacian~\cite{KL2020}) was proved. Surprisingly, differently from the Laplacian, in~\cite{KLS2021} the authors also showed the failure of uniform Sobolev and Carleman inequalities for the Lamé operator. In~\cite{KKLS2021} it was proved that if spacial lower-order  perturbations are replaced by temporal ones, \emph{i.e.} if one considers the damped equation, then those estimates become available again. In~\cite{B_F_R_V_V} the authors generalized the results in~\cite{B_F-G_P-E_R_V} proving Agmon-H\"ormander type estimates of the Navier equation when this is perturbed by small 0-th order matrix-valued potential. From these results Strichartz estimates for the evolution equation followed (in the same manner as for classical wave equation, see~\cite{BPST2003,BPST2004}). These Strichartz estimates were then generalized in~\cite{KKS2021,KKLS2021}. In particular in~\cite{KKLS2021} the endpoint case is deduced. 

The Navier equation got also attention of the inverse problem's community. In particular, inverse scattering was studied in~\cite{BFPRV2012,BFPRV2018}, whereas inverse boundary problems were considered in~\cite{NU1994,ER2002,BFV2014,IY2015,BFPRV2018}. Boundary determination of Lamé parameters has been studied in~\cite{CLLZ2021,LN2017,Tartar1997}.

\medskip
\noindent
In this paper we are interested in spectral properties of the damped elastic wave equation~\eqref{eq:damped}, or equivalently of the elastic wave operator~\eqref{eq:damped-operator}. More precisely, we aim at deducing quantitative bounds on the location of the point spectrum of $A_a^\ast$ in terms of suitable norms of the damping coefficient. In order to do that we establish a correspondence (see Lemma~\ref{lemma:correspondence}) between the eigenvalue problem associated to~\eqref{eq:damped-operator} and the one corresponding to suitable Lamé operators with non self-adjoint perturbations, that is operators of the form
\begin{equation}\label{eq:perturbed_Lame}
	-\Delta^\ast + V,
\end{equation} 
where $V$ denotes the operator of multiplication by a (possibly) non hermitian matrix-valued function $V\colon \R^d \to \C^{d\times d}$. 

The study of the spectrum of~\eqref{eq:perturbed_Lame} has already a bibliography. It is well known that the free Lamé operator $-\Delta^\ast$ is self-adjoint on $H^2(\R^d)^d$ and $\sigma(-\Delta^\ast)=\sigma_{\textup{ac}}(-\Delta^\ast)=[0,\infty).$ It is a natural question~\cite{Cossetti2017,Cossetti2019,CCF2021,
KLS2021} to ask whether and how the spectrum changes under 0th-order perturbations, \emph{i.e.} considering the operator~\eqref{eq:perturbed_Lame}. In~\cite{Cossetti2017}, adapting to the elasticity setting the method of multipliers developed for non self-adjoint Schrödinger operators in~\cite{FKV2018} (see also~\cite{F_K_V2} for similar problems on the plane), the author showed that the point spectrum of the perturbed Lamé operator~\eqref{eq:perturbed_Lame} remains empty (as in the free case) under suitable variational small perturbations (inverse-square Hardy potential with small coupling constant is covered). Later, in~\cite{CCF2021} we showed that full spectral stability, \emph{i.e.} $\sigma(-\Delta^\ast +V)=\sigma(-\Delta^\ast)=[0,\infty),$ can be proved in three dimensions $d=3$ under perturbations which satisfy a smallness condition of Hardy-type (see~\cite[Thm.~1.4]{CCF2021}). Focusing on the point spectrum only, if no stability can be proved a priori, an interesting question is related to provide quantitative bounds on the location in the complex plane of this part of the spectrum which, in the perturbed setting, is possibly no longer empty. In this direction, some preliminary result valid for the \emph{discrete} spectrum can be found in~\cite{Cossetti2019} (see also~\cite{KLS2021}). Later, these results have been extended in~\cite{CCF2021} to cover \emph{embedded} eigenvalues as well. More precisely in~\cite{CCF2021} the following result was proved.
\begin{theorem}[Thm.~1.1,~\cite{CCF2021}]\label{thm:Lp-result}
	Let $d\geq 2,$ $0<\gamma\leq 1/2$ if $d=2$ and $0\leq \gamma\leq 1/2$ if $d\geq 3$ and $V\in L^{\gamma + \frac{d}{2}}(\R^d;\C^{d\times d}).$ Then there exists a universal constant $c_{\gamma, d, \lambda, \mu}>0$ independent on $V$ such that
	\begin{equation}\label{eq:thesis}
		\sigma_\textup{p}(-\Delta^\ast + V)
		\subset \left\{z\in \C\colon |z|^\gamma \leq c_{\gamma, d, \lambda, \mu} \| V \|_{L^{\gamma + \frac{d}{2}}(\R^d)}^{\gamma + \frac{d}{2}}\right \}. 
	\end{equation}
\end{theorem}
\noindent
In the self-adjoint case, namely for real-valued perturbations, the result above holds for a larger class of indices $\gamma.$ More precisely, the following result holds true.
\begin{theorem}[Thm.~3.1,~\cite{Cossetti2019}]\label{thm:eigenvalue_bound_self}
	Let $d\geq 2,$ $\gamma>0$ if $d=2$ and $\gamma\geq 0$ if $d\geq3$ and $V\in L^{\gamma + \frac{d}{2}}(\R^d;\R).$ Then there exists a universal constant $c_{\gamma, d, \lambda, \mu}>0$ independent of $V$ such that any negative eigenvalue $z$ (if any) of the self-adjoint perturbed Lamé operator $-\Delta^\ast + VI_{\R^d}$ satisfies
	\begin{equation}\label{SA_bound_Lame}
		|z|^\gamma\leq 
		c_{\gamma, d, \lambda, \mu} \| V_- \|_{L^{\gamma + \frac{d}{2}}(\R^d)}^{\gamma + \frac{d}{2}},
	\end{equation}
	where $V_-$ is the negative part of $V,$ \emph{i.e.} $V_-(x):=\max\{-V(x),0\}.$
\end{theorem}
\noindent
Making use of Theorem~\ref{thm:Lp-result} and Theorem~\ref{thm:eigenvalue_bound_self} and the correspondence between the eigenvalue problem associated to the damped elastic wave operator and the one of the perturbed Lamé operator~\eqref{eq:perturbed_Lame} (see Lemma~\ref{lemma:correspondence} below) we shall prove the following two results valid in the self-adjoint and the non self-adjoint setting.
\begin{theorem}\label{thm:main-sa}
	Let $d\geq 2$ and assume $\gamma$ satisfies the hypotheses of Theorem~\ref{thm:eigenvalue_bound_self} and $a\in L^\infty(\R^d;\R).$ Then there exists a universal constant $c_{\gamma,d,\lambda, \mu}>0$ independent of the damping $a$ such that for any positive (respectively negative) eigenvalue $z$ of the damped elastic wave operator $A_a^\ast$ and $a_-\in L^{\gamma+ \frac{d}{2}}(\R^d)$ (respectively $a_+\in L^{\gamma+ \frac{d}{2}}(\R^d)$) satisfies
	\begin{equation}\label{eq:bound-sa}
		(\pm z)^{\gamma- \frac{d}{2}}\leq c_{\gamma,d,\lambda, \mu} \| a_{\mp} \|_{L^{\gamma + \frac{d}{2}}(\R^d)}^{\gamma + \frac{d}{2}},
	\end{equation}
\end{theorem}
\noindent
Setting $\gamma=d/2$ in~\eqref{eq:bound-sa}, the previous theorem provides sufficient condition on the size of the damping coefficient to guarantee absence of positive (respectively negative) eigenvalues.
\begin{corollary}\label{cor:absence-sa}
	If $d\geq 2$ and 
	\begin{equation*}
	c_{\frac{d}{2},d,\lambda, \mu} \|a_\mp\|_{L^d(\R^d)}^{d}<1,
	\end{equation*}
	then $A_a^\ast$ has no positive (respectively negative) eigenvalues.
\end{corollary}

\noindent
In the non self-adjoint setting we shall prove the following result. 
\begin{theorem}\label{thm:main-nsa}
	Let $d\geq 2$ and assume $\gamma$ satisfies the hypotheses of Theorem~\ref{thm:Lp-result} and $a\in L^\infty(\R^d;\C^{d\times d})$ is a (possibly) non hermitian matrix. Then there exists a universal constant $c_{\gamma,d,\lambda,\mu}>0$ independent of the damping $a$ such that
\begin{equation}\label{eq:bound-nsa}
	\sigma_{\textup{p}}(A_a^\ast)\subset 
	\Big\{
	z\in \C\colon |z|^{\gamma-\frac{d}{2}}
	\leq c_{\gamma, d,\lambda, \mu} \|a\|_{L^{\gamma+ \frac{d}{2}}}^{\gamma+ \frac{d}{2}}
	\Big\}.
\end{equation}	
\end{theorem}
\begin{remark}
	Notice that in the non self-adjoint case, due to the more restrictive class of indices for which Theorem~\ref{thm:Lp-result} is valid compared to Theorem~\ref{thm:eigenvalue_bound_self}, no analogous of Corollary~\ref{cor:absence-sa} holds true ($\gamma=d/2$ is not admissible).
\end{remark}

\medskip
\noindent
The main motivation behind our project relies on the following simple observation: the ellipticity condition~\eqref{eq:ellipticity} allows taking $\lambda+ \mu=0$ in the definition of the Lamé operator~\eqref{eq:Lame}. This choice turns the Lamé operator~\eqref{eq:Lame} into a vector Laplacian and consequently the damped elastic wave equation~\eqref{eq:damped} into a system of classical damped wave equations. For the (scalar) damped wave equation, results in the spirit of Theorem~\ref{thm:main-sa} and Theorem~\ref{thm:main-nsa} have been recently proved in~\cite{KK2020}. Thus, Theorem~\ref{thm:main-sa} and Theorem~\ref{thm:main-nsa} can be seen as a generalization of the results in~\cite[Thm.~1, Thm.~5 and Thm.~6]{KK2020} in the sense that they recover\footnote{the constants involved slightly differ due to the presence of the coefficient $\mu$ of the vector Laplacian and due to the vectorial form of the wave equation once $\lambda+ \mu=0$ in~\eqref{eq:damped}.} them when $\lambda+ \mu=0.$    

\medskip
\noindent
Theorem~\ref{thm:main-sa} and Theorem~\ref{thm:main-nsa} are not stated for $d=1,$ as a matter of fact the one dimensional case is rather special and it is treated separately. In $d=1$ the Lamé operator $-\Delta^\ast$ turns into a scalar differential operator, more precisely it is simply a multiple of the Laplacian
\begin{equation*}
	-\Delta^\ast
	:=-\mu \frac{d^2}{dx^2} -(\lambda + \mu) \frac{d^2}{dx^2}
	=-(\lambda + 2\mu) \frac{d^2}{dx^2}.
\end{equation*}
As a straightforward consequence of the celebrated result of Abramov, Aslanian and Davies for 1D-Schrödinger operators (see~\cite[Thm.~4]{AAD2001}), in~\cite{Cossetti2019} the following result for the one dimensional non self-adjoint Lamé operator was proved.
\begin{theorem}[Thm.~1.1,~\cite{Cossetti2019}]
\label{thm:1d}
	Let $d=1$ and $V\in L^1(\R;\C).$ Then
	\begin{equation*}
		\sigma_\textup{p}(-\Delta^\ast + V)\subset
		\Big\{z\in \C\colon |z|^{1/2}\leq \tfrac{1}{2\sqrt{\lambda+2\mu}} \|V\|_{L^1(\R)} \Big\}.
	\end{equation*} 
\end{theorem}
\begin{remark}
	We stress that Theorem 1.1 in~\cite{Cossetti2019} was stated only for eigenvalues outside the essential spectrum, namely for $z\in \C\setminus [0,\infty).$ Nevertheless, it is easy to show that embedded eigenvalues can be covered as well (see~\cite[Cor.~2.16]{DN2002}).
\end{remark}

\noindent
In the self-adjoint case, as an immediate consequence of the Lieb-Thirring inequalities (\hspace{-0.007cm}\cite{Laptev2012,LW2000}) valid for the Schrödinger operators, one has the following result.
\begin{theorem}\label{thm:1dLT}
	Let $d=1$ and $V_-\in L^1(\R;\R).$ Then
	\begin{equation}\label{eq:1dLT}
		\sigma_\textup{p}(-\Delta^\ast + V) \subset 
		\Big\{z\in \C\colon |z|^{1/2}\leq \tfrac{1}{2\sqrt{\lambda+2\mu}} \|V_-\|_{L^1(\R)} \Big\}.
	\end{equation}
\end{theorem}

\noindent
Theorem~\ref{thm:1d} and Theorem~\ref{thm:1dLT} together with Lemma~\ref{lemma:correspondence} below allow to deduce properties on the point spectrum of the one dimensional damped elastic wave operator $A_a^\ast.$
Differently from the higher dimensional setting, in $d=1$ Theorem~\ref{thm:1d} does not entail any quantitative bound on the location in the complex plane of the eigenvalues of $A_a^\ast,$ on the other hand it provides an explicit smallness condition on the size of the $L^1$-norm of the damping such that $A_a^\ast$ does not have eigenvalues. More precisely we have the following result.
\begin{theorem}\label{thm:1d-nsa}
	Let $d=1$ and $a\in L^1(\R;\C).$ If $\|a\|_{L^1(\R)}<2\sqrt{\lambda + 2\mu},$ then $\sigma_\textup{p}(A_a^\ast)=\varnothing.$ Moreover, the constant $2\sqrt{\lambda+ 2\mu}$ is optimal.
\end{theorem}

\noindent
In the self-adjoint situation it holds true a slightly different result compared to the ones introduced so far.
\begin{theorem}\label{thm:1d-sa}
Let $d=1$ and assume that $a$ is real-valued and satisfies 
\begin{equation}\label{eq:assumptions}
	\int_{\R} |x||a(x)|\, dx<\infty
	\qquad \text{and} \qquad
	\lim_{R\to \infty} \|a\|_{L^\infty(\R\setminus B_R(0))}=0.
\end{equation}
Let $z$ be a real eigenvalue of $A_a^\ast.$ If $z>0$ and $\int_{\R} a<-4\sqrt{\lambda + 2\mu}$ (or $z<0$ and $\int_{\R} a>4\sqrt{\lambda + 2\mu}$), then
\begin{equation*}
	|z|\geq (\lambda + 2\mu) \Bigg(\int_{\R} |x||a(x)|\, dx \Bigg)^{-1}.
\end{equation*}	
\end{theorem}
\noindent
Moreover the following quantitative bound on the location of eigenvalues holds.
\begin{theorem}\label{thm:1d-sa-bound}
	Let $d=1$ and assume that $a$ is real-valued and satisfies~\eqref{eq:assumptions}. Moreover, assume 
	\begin{equation*}
		|z|<(\lambda+ 2\mu)\Bigg(\int_{\R} |x||a(x)|\, dx \Bigg)^{-1}.
	\end{equation*}  
	If $z>0$ and $\int_\R a<0$ (respectively, $z<0$ and $\int_\R a>0$), then there exists exactly one $\alpha>0$ satisfying
	\begin{equation*}
 		2\Bigg(\int_\R a_-(x)\, dx\Bigg)^{-1}\leq \alpha \leq -4 \Bigg(\int_\R a(x)\, dx\Bigg)^{-1}
 		\quad \Bigg(\text{respectively, } 2\Bigg(\int_\R a_+(x)\, dx\Bigg)^{-1}\leq \alpha \leq 4 \Bigg(\int_\R a(x)\, dx\Bigg)^{-1} \Bigg) 
	\end{equation*}
	such that $z/\alpha$ is an eigenvalue of $A_{a}^\ast.$
\end{theorem}

\noindent
The rest of the paper is divided as follows. In the next Section we provide the proof of the preliminary Lemma~\ref{lemma:correspondence} establishing the correspondence between the eigenvalue problem associated to the damped elastic wave operator and the perturbed Lamé operator. Afterwards, in Section~\ref{sec:higher} we show the validity of Theorem~\ref{thm:main-sa} and Theorem~\ref{thm:main-nsa} which hold in higher dimension $d\geq 2.$ The one dimensional case, that is Theorem~\ref{thm:1d-nsa}-Theorem~\ref{thm:1d-sa-bound}, is treated separately in Section~\ref{sec:1d}.
\section{Proofs}
As a starting point we show how the eigenvalue problem associated to the damped elastic wave operator $A_a^\ast$ defined in~\eqref{eq:damped-operator} is related to the one of a perturbed Lamé operator of the form~\eqref{eq:perturbed_Lame}.
\begin{lemma}\label{lemma:correspondence}
Let $d\geq 1$ and assume $a\in L^\infty(\R^d;\C^{d\times d}).$ For every $z\in \C,$
\begin{equation*}
z\in \sigma_\textup{p}(A_a^\ast)
\quad  \Longleftrightarrow \quad 
-z^2\in \sigma_\textup{p}(-\Delta^\ast + za).
\end{equation*} 
\end{lemma}
\begin{proof}
	Assume $z\in \sigma_\textup{p}(A_a^\ast),$ then there exists a non-trivial $\Psi=(\psi_1,\psi_2)^T\in\mathcal{D}(A_a^\ast)$ such that $A_a^\ast \Psi=z\Psi.$ In other words, $\psi_1 \in H^2(\R^d)^d,$ $\psi_2\in \dot{H}^1(\R^d)^d$ and $\psi_2=z\psi_1,$ $\Delta^\ast \psi_1 -a \psi_2=z\psi_2.$ Plugging the first equation in the second one gives $-\Delta^\ast \psi_1 +za\psi_1=-z^2\psi_1.$ Since $\psi_1\neq 0,$ then $-z^2\in \sigma_\textup{p}(-\Delta^\ast +za).$ 
	Conversely, assume $-z^2\in  \sigma_\textup{p}(-\Delta^\ast +za),$ then there exists a non-trivial $\psi\in H^2(\R^d)^d$ such that $(-\Delta^\ast +za)\psi=-z^2 \psi.$ Defining $\Psi:=(\psi, z\psi)^T,$ then $\Psi \in \mathcal{D}(A_a^\ast)$ and 
		$(A_a^\ast \Psi)^T=(z \psi, \Delta^\ast \psi - za \psi)=z(\psi, z\psi)=z\Psi^T.$
	Therefore, $z\in \sigma_\textup{p}(A_a^\ast).$
 \end{proof}
 
 \begin{remark} \label{rmk:ev0}
From the validity of Lemma~\ref{lemma:correspondence}, one has that $0\notin \sigma_\textup{p}(A_a^\ast)$ as the spectrum of the unperturbed Lamé operator $-\Delta^\ast + 0 a=-\Delta^\ast$ is purely continuous.
 \end{remark}

\subsection{Higher dimensions $d\geq2:$ Proof of Theorem~\ref{thm:main-sa} and Theorem~\ref{thm:main-nsa}}\label{sec:higher}
With Lemma~\ref{lemma:correspondence} at hands we now show that Theorem~\ref{thm:main-sa} and Theorem~\ref{thm:main-nsa} are consequence of Theorem~\ref{thm:eigenvalue_bound_self} and Theorem~\ref{thm:Lp-result}, respectively. 
 
\begin{proof}[Proof of Theorem~\ref{thm:main-sa}]
 From Lemma~\ref{lemma:correspondence} we know that $z\in \sigma_\textup{p}(A_a^\ast)$ if and only if $-z^2\in \sigma(-\Delta^\ast + za).$ From Theorem~\ref{thm:eigenvalue_bound_self} there exists $c_{\gamma,d,\lambda, \mu}>0$ such that 
 \begin{equation}\label{eq:prel}
 	|z|^{2\gamma}\leq c_{\gamma,d,\lambda,\mu} \|(za)_-\|_{L^{\gamma + \frac{d}{2}}(\R^d)}^{L^{\gamma + \frac{d}{2}}},
 \end{equation}
 where $(za)_-$ is the negative part of $za,$ \emph{i.e.} $(za)_-= z a_+$ if $z\in (-\infty, 0)$ and $(za)_-=za_-$ if $z\in(0,\infty).$ Using this fact in~\eqref{eq:prel} and dividing both sides of~\eqref{eq:prel} by $|z|^{\gamma + d/2}$ ($z\neq 0,$ see Remark~\ref{rmk:ev0}) we obtain~\eqref{eq:bound-sa}.
\end{proof}
 
\medskip
\noindent 
Now we consider the non self-adjoint situation. 
\begin{proof}[Proof of Theorem~\ref{thm:main-nsa}]
The proof of Theorem~\ref{thm:main-nsa} is analogous to the one of Theorem~\ref{thm:main-sa}. Let $z\in \sigma_\textup{p}(A_a^\ast),$ then by Lemma~\ref{lemma:correspondence} $-z^2\in \sigma_\textup{p}(-\Delta^\ast + za).$ Using the eigenvalue bound~\eqref{eq:thesis} then one has
\begin{equation*}
|z|^{2\gamma}\leq c_{\gamma, d,\lambda, \mu} |z|^{\gamma + \frac{d}{2}} \|a\|_{L^{\gamma + \frac{d}{2}}(\R^d)}^{\gamma + \frac{d}{2}},
\end{equation*}
which gives~\eqref{eq:bound-nsa} and concludes the proof. 
\end{proof}

\subsection{1D: Proof of Theorem~\ref{thm:1d-nsa}, Theorem~\ref{thm:1d-sa} and Theorem~\ref{thm:1d-sa-bound}} \label{sec:1d}
We start considering the self-adjoint situation. Let $z\in \R$ and let $\{\lambda_n^\ast(za)\}_{n=1}^N$ denote the sequence of eigenvalues of $-\Delta^\ast + za,$ then the following preliminary lemma on the sum of the square root of the eigenvalues holds. 
\begin{lemma}
	Let $d=1.$ Then 
	\begin{equation}\label{eq:sum}
		\sum_{n=1}^N|\lambda_n^\ast(za)|^{1/2}\geq -\frac{z}{4\sqrt{\lambda+ 2\mu}} \int_\R a(x)\, dx.
	\end{equation}
	Moreover if   
		$\int_{\R} |x||a(x)|\, dx<\infty,$ then
	the following bound on the number $N$ of eigenvalues $\lambda_n^\ast(za)$
	\begin{equation}\label{eq:NLame}
		N\leq 1 + \frac{|z|}{\lambda + 2\mu} \int_{\R} |x||a(x)|\, dx
	\end{equation}
	holds.
\end{lemma}
\begin{proof}
	If $\lambda_n^\ast(za)$ is an eigenvalue of $-\Delta^\ast + za,$ then there exists $\psi\in H^2(\R)$ such that $-(\lambda + 2\mu)
	\Delta\psi+ za\psi=\lambda_n^\ast(za)\psi$ or, equivalently, 
	\begin{equation}\label{eq:Schro}
		\Big(-\Delta+ \frac{za}{\lambda + 2\mu}\Big)\psi=\frac{\lambda_n^\ast(za)}{\lambda + 2\mu}\psi.
	\end{equation}  
	Denoting by $\lambda_n(V)$ the eigenvalues of the Schrödinger operator $-\Delta+ V,$ then we conclude that $\lambda_n^\ast(za)$ is an eigenvalue of $-\Delta^\ast + za$ if and only if there exists $n\in \N$ such that $\lambda_n^\ast(za)$ is a multiple of an eigenvalue $\lambda_n(za/(\lambda + 2\mu))$ of the Schrödinger operator $-\Delta + za/(\lambda+ 2\mu),$ more precisely $\lambda_n(za/(\lambda + 2\mu))= \lambda_n^\ast(za)/(\lambda + 2\mu).$ In particular the number of eigenvalues coincides. The Buslaev-Faddeev-Zakharov trace formula (\emph{cf.}~\cite{FZ2016}) for 1D-Schrödinger operator $-\Delta + V$ states that
	\begin{equation*}
		\sum_{n=1}^N |\lambda_n(V)|^{1/2}\geq -\frac{1}{4} \int_{\R} V(x)\, dx,
	\end{equation*}
	this and the correspondence above give immediately~\eqref{eq:sum}.
	
	The Bargmann bound~\cite[Pb.~22]{RSIV} provides a control from above of the number of eigenvalues of the 1D-Schrödinger operator $-\Delta + V$ under the condition $\int_\R |x||V(x)\, dx<\infty.$ More precisely,
	\begin{equation}\label{eq:NSchro}
		N\leq 1 + \int_\R |x||V(x)|\, dx.
	\end{equation}
	Similarly as above (that is using the correspondence between eigenvalues of the Lamé operator $-\Delta^\ast + za$ and of the Schrödinger operator $-\Delta + za/(\lambda+ 2\mu)$) from~\eqref{eq:NSchro} one easily gets~\eqref{eq:NLame}. This concludes the proof.
\end{proof}
\begin{proof}[Proof of Theorem~\ref{thm:1d-sa}]
	Let $z$ be a real eigenvalue of $A_a^\ast,$ in order to prove Theorem~\ref{thm:1d-sa} we will show that if $|z|<(\lambda + 2\mu)\Big(\int_{\R} |x||a(x)|\, dx \Big)^{-1}$ then $\int_{\R} a\geq -4\sqrt{\lambda + 2\mu}$ for $z>0$ and $\int_{\R}a \leq 4\sqrt{\lambda+ 2\mu}$ for $z<0.$ First of all notice that~\eqref{eq:sum} is non-trivial only if $z\int_{\R} a(x)\, dx<0.$ This last condition, in particular is known to be a sufficient condition which guarantees that $\inf \sigma(-\Delta^\ast + za)<0.$ From the decay assumption~\eqref{eq:assumptions}, then it follows that $-\Delta^\ast + za$ posses at least one negative eigenvalue. From the upper bound~\eqref{eq:NLame} it follows that if $|z|<(\lambda + 2\mu)\Big(\int_{\R} |x||a(x)|\, dx \Big)^{-1}$ then $-\Delta^\ast + za$ has exactly one negative eigenvalue $\lambda_1^\ast(za).$ Thus, from~\eqref{eq:sum} and the correspondence in Lemma~\ref{lemma:correspondence} one has   
	\begin{equation}\label{eq:estimate1d}
		|z|=|\lambda_1(za)|^{1/2}\geq -\frac{z}{4\sqrt{\lambda + 2\mu}}\int_\R a(x)\, dx.
	\end{equation}
	This implies $\int_\R a(x)\, dx\geq -4\sqrt{\lambda+ 2\mu}$ for $z>0$ and $\int_\R a(x)\, dx\leq 4\sqrt{\lambda+ 2\mu}$ for $z<0.$
\end{proof}

\begin{proof}[Proof of Theorem~\ref{thm:1d-sa-bound}]
	From the hypotheses, as above, one has that $-\Delta^\ast + za$ posses exactly one negative eigenvalue. The Lieb-Thirring type bound~\eqref{eq:1dLT} in Theorem~\ref{thm:1dLT} and the estimate in~\eqref{eq:estimate1d} give
	\begin{multline*}
		-\frac{z}{4\sqrt{\lambda + 2\mu}} \int_{\R} a(x)\, dx\leq |\lambda_1(za)|^{1/2}\leq \frac{z}{2\sqrt{\lambda + 2\mu}} \int_{\R}a_-(x)\, dx,\\
		\qquad \Big(\text{respectively }
		-\frac{z}{4\sqrt{\lambda + 2\mu}} \int_{\R} a(x)\, dx\leq |\lambda_1(za)|^{1/2}\leq \frac{|z|}{2\sqrt{\lambda + 2\mu}} \int_{\R}a_+(x)\, dx \Big).
	\end{multline*}
	Using the correspondence in Lemma~\ref{lemma:correspondence} the result follows.
\end{proof}

\begin{proof}[Proof of Theorem~\ref{thm:1d-nsa}]
	If $z\in \C$ is an eigenvalue of $A_a^\ast,$ then from Lemma~\ref{lemma:correspondence} $-z^2\in \sigma_\textup{p}(-\Delta^\ast + za).$ Thus, from Theorem~\ref{thm:1d} we have
	\begin{equation*}
		|z|\leq \frac{1}{2\sqrt{\lambda+ 2\mu}} |z|\|a\|_{L^1(\R)}.
	\end{equation*}
	Dividing by $|z|,$ which cannot be zero (see Remark~\ref{rmk:ev0}), one has
		$1\leq \tfrac{1}{2\sqrt{\lambda+ 2\mu}} \|a\|_{L^1(\R)}.$
	If the $L^1$-norm of $a$ is small, namely if $\|a\|_{L^1(\R)}<2\sqrt{\lambda+ 2\mu},$ then we get a contradiction. Thus, $\sigma_\textup{p}(A_a^\ast)=\varnothing.$ The optimality of the result can be proved as in~\cite[Thm.~4]{KK2020}.
\end{proof}

\providecommand{\bysame}{\leavevmode\hbox to3em{\hrulefill}\thinspace}


\end{document}